\documentclass[a4paper,12pt]{amsart}

\usepackage{epsfig}
\usepackage{amsthm,amsfonts}
\usepackage{amssymb,graphicx,color}
\usepackage[all]{xy}
\usepackage{cancel} 
\usepackage{verbatim}
\usepackage{hyperref}

\newtheorem{theorem}{Theorem}[section]
\newtheorem*{theorem*}{Theorem}
\newtheorem{lemma}[theorem]{Lemma}
\newtheorem{corollary}[theorem]{Corollary}

\newtheorem{definition}[theorem]{Definition}

\newtheorem{remark}[theorem]{Remark}

\newcommand{\R}{\mathbb{R}}

\newcommand{\C}{\mathbb{C}}

\newcommand{\N}{\mathbb{N}}

\newcommand{\Sp}{\mathbb{S}}

\newcommand{\x}{\textbf {x}}

\newcommand{\T}{\widetilde}

\begin{document}

\title[bi-Lip. homeomorphic subanalytic sets have bi-Lip. homeomorphic tangent cones]
{bi-Lipschitz homeomorphic subanalytic sets have bi-Lipschitz homeomorphic tangent cones}

\author{J. E. Sampaio}
\address{J. E. Sampaio - 
Departamento de Matem\'atica, Universidade Federal do Cear\'a Av.
Humberto Monte, s/n Campus do Pici - Bloco 914, 60455-760
Fortaleza-CE, Brazil} 
 \email{edsonsampaio@mat.ufc.br}


\begin{abstract}
We prove that if there exists a bi-Lipschitz homeomorphism (not necessarily subanalytic) between two subanalytic sets, then their tangent cones are bi-Lipschitz homeomorphic. As a consequence of this result, we show that any Lipschitz regular complex analytic set, i.e any complex analytic set which is locally bi-lipschitz homeomorphic to an Euclidean ball must be smooth. Finally, we give an alternative proof of S. Koike and L. Paunescu's result about the bi-Lipschitz invariance of directional dimensions of subanalytic sets.
\end{abstract}

\maketitle
\setcounter{section}{-1}
\section{Introduction}
Geometry and topology of tangent cones to algebraic or analytic sets are very important in several questions of Singularity Theory.  The tangent cones at singular points generalize the notion of tangent spaces at smooth points. In \cite{Whitney:1965a} and \cite{Whitney:1965b} (see also \cite{Whitney:1972}) Whitney established several notions of tangent cones and proved some equivalences for these definitions in the case of complex analytic sets. Modern Real Algebraic Geometry also studies the tangent cones of semialgebraic and subanalytic sets.

In 1971, in his famous paper about open questions in Singularity Theory \cite{Zariski:1971}, Zariski asked the following question: does the topological equivalence of complex algebraic sets imply the topological equivalence of their tangent cones? This question was answered negatively by Fernandez de Bobadilla \cite{Bobadilla:2005}. 
 
In this paper, we study the behavior of tangent cones under bi-Lipschitz map. In Section \ref{section:cones}, we recall the definition of tangent cone and we list some of their properties. In Section \ref{section:mainresult}, we show the main result of the paper, namely: if two subanalytic  sets are bi-Lipschitz homeomorphic, then their tangent cones are also bi-Lipschitz homeomorphic. This was proved by Bernig and Lytchak \cite{BL} under the additional assumption that the bi-Lipschitz homeomorphism is also subanalytic (see also the paper of Birbrair, Fernandes and Neumann \cite{BFN} for an extremely simple proof). This result gives a positive answer to Zariski's question for bi-Lipschitz homeomorphism. 

Finally, in Section \ref{section:applications}, we give some applications of the main result. A generalization of Theorem 3.2 of \cite{BFLS} is proved. Namely, Lipschitz regular (without hypotheses of subanalyticity) complex analytic sets are actually smooth. This result can be considered as an analog of the Theorem of Mumford \cite{M} about the smoothness of topologically regular complex algebraic surfaces. At the end of this paper, we also present a short proof of the result of Koike and Paunescu \cite{Koike:2009} on the bi-Lipschitz invariance of the directional dimension.

\bigskip

\noindent{\bf Acknowledgements}. I wish to thank my thesis advisor Alexandre Fernandes and Lev Birbrair by incentive in this research and also for their support and help. I also wish to thank Vincent Grandjean for corrections and suggestions in writing this article. The result of the paper is part of my PhD thesis at Universidade Federal do Cear\'a.

\section{Preliminaries}\label{section:cones}
\begin{definition}
Let $A\subset \R^n$ be a subanalytic set such that $x_0\in \overline{A}$.
We say that $v\in \R^n$ is a tangent vector of $A$ at $x_0\in\R^n$ if there are a sequence of points $\{x_i\}\subset A\setminus \{x_0\}$ tending to $x_0$ and sequence of positive real numbers $\{t_i\}$ such that 
$$\lim\limits_{i\to \infty} \frac{1}{t_i}(x_i-x_0)= v.$$
Let $C(A,x_0)$ denote the set of all tangent vectors of $A$ at $x_0\in \R^n$. We call $C(A,x_0)$ the {\bf tangent cone} of $A$ at $x_0$.
\end{definition}
Notice that $C(A,x_0)$ corresponds to the cone $C_3(A,x_0)$ defined by Whitney \cite{Whitney:1965a,Whitney:1965b,Whitney:1972}. 

\begin{remark} {\rm It follows from the Curve Selection Lemma for subanalytic sets that, if $A\subset \R^n$ is a subanalytic set and $x_0\in \overline{A}$ then the following holds true }
$$C(A,x_0)=\{v;\, \exists\, \alpha:[0,\varepsilon )\to \R^n\,\, \mbox{s.t.}\,\, \alpha(0)=x_0,\, \alpha((0,\varepsilon ))\subset A\,\, \mbox{and}\,\, \alpha(t)-x_0=tv+o(t)\}.$$
\end{remark}

\begin{remark}
{\rm If $A\subset \C^n$ is a complex analytic set such that $x_0\in A$ then $C(A,x_0)$ is the zero set of a set of complex homogeneous polynomials (see \cite{Whitney:1972}, Theorem 4D). In particular, $C(A,x_0)$ is the union of complex line passing for $x_0$.}
\end{remark}
\begin{remark}
{\rm All the sets considered in the paper are supposed to be equipped with the Euclidean metric. In particular, when we say that $f:X\subset\R^n\rightarrow\R^m$ is a Lipschitz map, we mean that there exists a constant $K>0$ such that }
$$\| f(x_1)-f(x_2)\|\leq K \| x_1-x_2\| \ \forall \ x_1,x_2\in X.$$
\end{remark}

\section{Main result}\label{section:mainresult}
\begin{lemma}\label{extension}
Let $X,Y\subset \R^n$ be subanalytic sets. If $X$ and $Y$ are bi-Lipschitz homeomorphic, then there is a bi-Lipschitz homeomorphism $\varphi:\R^{2n}\to \R^{2n}$ such that $\varphi(X\times\{0\})=\{0\}\times Y$.
\end{lemma}
\begin{proof}
Let $\phi: X\to Y$ be a bi-Lipschitz homeomorphism. By McShane-Whitney-Kirszbraun's Theorem (see \cite{Mcshane:1934}, \cite{Whitney:1934} and \cite{Kirszbraun:1934}) there exists $\T{\phi}: \R^n\to \R^n$ a Lipschitz map such that $\T{\phi}|_X=\phi$ and $\T{\psi}: \R^n\to \R^n$ another Lipschitz map such that $\T{\psi}|_Y=\phi^{-1}$.
Let us define $\varphi, \psi: \R^n\times \R^n\to \R^n\times \R^n$ as follows:
$$
\varphi(x,y)=(x-\T{\psi}(y+\T{\phi}(x)),y+\T{\phi}(x))
$$
and
$$
\psi(z,w)=(z+\T{\psi}(w), w-\T{\phi}(z+\T{\psi}(w))).
$$
Since $\varphi$ and $\psi$ are composition of Lipschitz maps, they are also Lipschitz maps. 

Next, we show that $\psi=\varphi^{-1}$.  In fact, if $(x,y)\in \R^n\times \R^n$ then
$$
\begin{array}{l}
\psi(\varphi(x,y))=\psi(x-\T{\psi}(y+\T{\phi}(x)),y+\T{\phi}(x))\\
= (x-\T{\psi}(y+\T{\phi}(x))+\T{\psi}(y+\T{\phi}(x)),y+\T{\phi}(x)-\T{\phi}(x-\T{\psi}(y+\T{\phi}(x))+\T{\psi}(y+\T{\phi}(x)))\\
= (x,y+\T{\phi}(x)-\T{\phi}(x))\\
= (x,y),
\end{array}
$$
and if $(z,w)\in \R^n\times \R^n$ then
$$
\begin{array}{l}
\varphi(\psi(z,w))=\varphi(z+\T{\psi}(w), w-\T{\phi}(z+\T{\psi}(w)))\\
= (z+\T{\psi}(w)-\T{\psi}(w-\T{\phi}(z+\T{\psi}(w))+\T{\phi}(z+\T{\psi}(w))),w-\T{\phi}(z+\T{\psi}(w))+\T{\phi}(z+\T{\psi}(w)))\\
= (z+\T{\psi}(w)-\T{\psi}(w),w)\\
= (z,w).
\end{array}
$$
Therefore $\psi=\varphi^{-1}$. Finally, it is clear that  $\varphi(X\times \{0\})=\{0\}\times Y$.
\end{proof}
Now we can state and prove our mean result.
\begin{theorem}\label{equivcone}
Let $X,Y\subset \R^n$ be subanalytic sets. If the germs $(X,x_0)$ and $(Y,y_0)$ are bi-Lipschitz homeomorphic, then $(C(X,x),x)$ and $(C(Y,y),y)$ are also bi-lipschitz homeomorphic.
\end{theorem}
\begin{proof}
We can assume that $x_0=y_0=0$. By Lemma \ref{extension}, we can suppose that there exists a bi-Lipschitz map $\varphi: \R^m\to\R^m$ such that $\varphi(X)=Y$. Let $K>0$ be a constant such that
\begin{equation}
\frac{1}{K}\|x-y\|\leq \|\varphi(x)-\varphi(y)\|\leq K\|x-y\|, \quad \forall x,y\in \R^m.
\end{equation}
Let $\psi=\varphi^{-1}$. Let us define two sequences of maps: for each $n\in\N$, we define the mappings $\varphi_n:\overline{B_1^m}\to \R^m$ and $\psi_n:\overline{B_K^m}\to \R^m$ as follows 
$$\varphi_n(v)=n\varphi(\frac{v}{n}) \quad \mbox{and} \quad \psi_n(v)=n\psi(\frac{v}{n}).$$ 

Observe that for any $n\in \N$, we have
$$
\frac{1}{K}\|u-v\|\leq \|\varphi_n(u)- \varphi_n(v)\|\leq K\|u-v\|, \quad \forall u,v\in \overline{B_1^m},
$$
and 
$$
\frac{1}{K}\|u-v\|\leq \|\psi_n(u)-\psi_n(v)\|\leq K\|u-v\|, \quad \forall u,v\in\overline{B_K^m}.
$$
Then, by Arzela-Ascoli's Theorem, there exists a subsequence $\{n_j\}\subset \N$ and mappings $d\varphi:\overline{B_1^m}\to \R^m$ and  $d\psi: \overline{B_K^m}\to \R^m$ such that $ \varphi_{n_j}\to d\varphi$ and $ \psi_{n_j}\to d\psi$ uniformly as $j\to \infty $. Clearly
$$
\frac{1}{K}\|u-v\|\leq \|d\varphi(u)-d\varphi(v)\|\leq K\|u-v\|
$$
and
$$
\frac{1}{K}\|z-w\|\leq \|d\psi(z)-d\psi(w)\|\leq K\|z-w\|.
$$
Let $U=d\varphi(B_1^m)$. Since $d\varphi$ is a continuous map and injective, $U$ is a open.

\noindent{\bf Claim. } $d\psi\circ d\varphi=id_{B_1^m}$ and $d\varphi\circ d\psi|_U=id_{U}$.

Let $v\in C(X,0)\cap B_1^m$ and $w=d\varphi(v)=\lim \limits_{j\to \infty }\frac{\varphi(t_jv)}{t_j}$ with $t_j=\frac{1}{n_j}$. 
$$
\begin{array}{lllll}
\|d\psi(w)-v\|&=&\big\|\lim \limits_{j\to \infty }\frac{\psi(t_jw)}{t_j} - v\big\| &=&\lim \limits_{j\to \infty }\big\|\frac{\psi(t_jw)}{t_j} - \frac{t_jv}{t_j}\big\|\\
				    &=&\lim \limits_{j\to \infty }\frac{1}{t_j}|\psi(t_jw) - t_jv\| &=&\lim \limits_{j\to \infty }\frac{1}{t_j}\|\varphi^{-1}(t_jw) - \varphi^{-1}(\varphi(t_jv))\|\\
				    &\leq &\lim \limits_{j\to \infty }\frac{K}{t_j}\|t_jw - \varphi(t_jv)\| &\leq &\lim \limits_{j\to \infty }K\big\|w - \frac{\varphi(t_jv)}{t_j}\big\|\\
					&=&0.& &
\end{array}
$$

Hence, $d\psi(w)=d\psi(d\varphi(v))=v$ for all $v\in B_1^m$, i.e., $d\psi\circ d\varphi=id_{B_1^m}$. Similarly, $d\varphi\circ d\psi|_U=id_{U}$.

\noindent{\bf Claim. } $d\varphi(C(X,0)\cap B_1^m)\subset C(Y,0)$ and $d\psi(C(Y,0)\cap d\varphi(B_1^m))\subset C(X,0)$.

Indeed, let $v\in C(X,0)\cap B_1^m$, then there exists $\alpha: [0,\varepsilon )\to X$ such that $\alpha(t)=tv+o(t)$. Thus $\varphi(\alpha(t))\in Y$, for all $t\in [0,\varepsilon )$ and since $\varphi$ is Lipschitz, we have $\varphi(\alpha(t))=\varphi(tv)+o(t)$. On the other hand, by definition of the map $d\varphi$, we get $\varphi(t_jv)=t_jd\varphi(v)+o(t_j)$, hence
$$
d\varphi (v)=\lim \limits_{j\to \infty }\varphi_{n_j}(v)=\lim \limits_{j\to \infty }\frac{\varphi(t_jv)}{t_j}=\lim \limits_{j\to \infty }\frac{\varphi(\alpha(t_j))}{t_j}\in C(Y,0).
$$
Thus $d\varphi(C(X,0)\cap B_1^m)\subset C(Y,0)$. 

Likewise, we have $d\psi(C(Y,0)\cap U)\subset C(X,0)$. Therefore $d\varphi:C(X,0)\cap B_1^m\to C(Y,0)\cap U$ is a bi-Lipschitz map.
\end{proof}
\begin{remark}
{\rm It is worth observing that, it is possible to adapt the proof above in order to get  the map $d\varphi$ being a global bi-Lipschitz homeomorphism on $\R^m$ such that $d\varphi(C(X,0))=C(Y,0)$. In particular, cones which are  bi-Lipschitz homeomorphic as germs at their vertices are globally bi-Lipschitz homeomorphic.}
\end{remark}

\section{Applications}\label{section:applications}
\begin{definition} A subanalytic subset $X\subset\R^n$ is called {\bf Lipschitz regular}  (respectively {\bf subanalytically Lipschitz regular}) at $x_0\in X$ if there is an open neighborhood $U$ of $x_0$ in $X$ which is bi-Lipschitz homeomorphic (respectively subanalytic bi-Lipschitz homeomorphic) to an Euclidean ball. 
\end{definition}
In the joint work of the author \cite{BFLS}, we define the notion of Lipschitz regular complex analytic sets as the sets being subanalytically Lipschitz regular as in the definition above.
\begin{lemma}\label{conesmooth} 
Let $X\subset\C^n$ be a complex analytic subset. Let $x_0\in X$ be such that the reduced tangent cone $C(X,x_0)$ is a linear subspace of $\C^n$. If $X$ is Lipschitz regular at $x_0\in X$, then $x_0$ is a smooth point of $X$.
\end{lemma}
\begin{proof}
It follows directly from the Lemma 3.3 and Proposition 2.1 in \cite{BFLS}.
\end{proof}

\begin{theorem} Let $X\subset\C^n$ be a complex analytic set. If $X$ is Lipschitz regular at $x_0\in X$, then $x_0$ is a smooth point of $X$.
\end{theorem}

\begin{proof} Let us suppose that $X$ is Lipschitz regular at $x_0$. Let $C(X,x_0)$ be the tangent cone of $X$ at $x_0$. Let $\phi\colon U\rightarrow B$ be a bi-Lipschitz homeomorphism where $U$ is a neighborhood of $x_0$ in $X$ and $B\subset\R^N$ is an Euclidean ball centered at origin $0\in\R^N$. By Theorem \ref{equivcone}, $(C(X,x_0),x_0)$ is bi-Lipschitz homeomorphic the $(C(B,0),0)=(\R^N,0)$.
By the Theorem of Prill (see \cite{P}), the reduced cone $C(X,x_0)$ is a linear subspace of $\C^n$. We finish the proof using Lemma \ref{conesmooth}.
\end{proof}
\begin{corollary}[\cite{BFLS}, Theorem 3.2]
Let $X\subset\C^n$ be a complex algebraic set. If $X$ is subanalytically Lipschitz regular at $x_0\in X$, then $x_0$ is a smooth point of $X$.
\end{corollary}

From now, we approach  a result of Koike and Paunescu about the bi-Lipschitz invariance of the dimension of directional set of subanalytic germs. First of all, let us introduce the definition of directional set.
  
\begin{definition}
Let $A\subset \R^n$ be a set-germ at $0\in \R^n$ such that $0\in \overline{A}$. 
We say that $\x\in \Sp^{n-1}$ is a direction of $A$ at $0\in\R^n$ if there is a sequence of points $\{x_i\}\subset A\setminus \{0\}$ converging to $0\in \R^n$ such that $\frac{x_i}{\|x_i\|}\to \x$ as $i\to \infty $. 
Let $D(A)$ denote the set of all directions of $A$ at $0\in \R^n$.
\end{definition}
\begin{remark}\label{remarkcone}
{\rm The link of the cone $C(A,0)$ is $D(A)$, i.e., $D(A)=C(A,0)\cap \Sp^{n-1}$.}
\end{remark}

\begin{theorem}[\cite{Koike:2009}, Main Theorem]
Let $A,B\subset \R^n$ be subanalytic set-germs at $0\in \R^n$ such that $0\in \overline{A}\cap \overline{B}$, and $h:(\R^n,0)\to (\R^n,0)$ be a bi-Lipschitz homeomorphism. Suppose that $h(A)$, $h(B)$ are also subanalytic. The following equality of dimensions holds 
$$
\dim (D(h(A))\cap D(h(B)))=\dim (D(A)\cap D(B)).
$$
\end{theorem}
\begin{proof}
By Theorem \ref{equivcone}, there exists a bi-Lipschitz homeomorphism $dh:B_1^n\to U$ with $U:=dh(B_1^n)$ being an open set of $\R^n$ and such that
$$dh(C(A,0)\cap B_1^n)=C(h(A),0)\cap U$$
and
$$ dh(C(B,0)\cap B_1^n)=C(h(B),0)\cap U.$$ 
Since $dh$ is a homeomorphism, $dh$ preserves intersection and then 
\begin{eqnarray*}
dh(C(A,0)\cap C(B,0)\cap B_1^n)&=& dh(C(A,0)\cap B_1^n)\cap dh(C(B,0)\cap B_1^n)\\
							   &=& C(h(A),0)\cap C(h(B),0)\cap U
\end{eqnarray*}
and since $dh$ preserves dimension, we obtain that
$$\dim (C(h(A),0)\cap C(h(B),0)\cap B_1^n)=\dim (C(A,0)\cap C(B,0)\cap U).$$
But 
$$\dim (C(h(A),0)\cap C(h(B),0)\cap B_1^n)=\dim (C(h(A),0)\cap C(h(B),0))$$ 
and 
$$\dim (C(A,0)\cap C(B,0)\cap U)=\dim (C(A,0)\cap C(B,0)),$$
since $B_1^n$ and $U$ are open and the intersections above are not empty. Then 
$$\dim (C(h(A),0)\cap C(h(B),0))=\dim (C(A,0)\cap C(B,0)),$$
and therefore $\dim (D(h(A))\cap D(h(B)))=\dim (D(A)\cap D(B))$.
\end{proof}

\end{document}